\documentclass[a4paper,12pt]{amsart}
\usepackage{amsmath,amsthm,amssymb}
\usepackage{amsfonts}
\usepackage{txfonts}
\usepackage{latexsym}
\usepackage[dvipdfmx]{graphics,color}
\usepackage{eucal}
\usepackage{mathrsfs}
\theoremstyle{plain}
\newtheorem{theorem}{Theorem}[section]
\newtheorem{proposition}[theorem]{Proposition}
\newtheorem{lemma}[theorem]{Lemma}
\newtheorem{corollary}[theorem]{Corollary}

\theoremstyle{definition}
\newtheorem{definition}[theorem]{Definition}

\makeatletter
\renewenvironment{proof}[1][\proofname]{\par
  \normalfont
  \topsep6\p@\@plus6\p@ \trivlist
  \item[\hskip\labelsep{\bfseries #1}\@addpunct{\bfseries.}]\ignorespaces
}{%
  \endtrivlist
}
\renewcommand{\proofname}{proof}
\theoremstyle{remark}
\newtheorem{remark}[theorem]{Remark}

\numberwithin{equation}{section}

\title{Algebraic $K$-theory and algebraic cobordism of almost mathematics}
\date{\today} 
\author{Yuki Kato}
\address{National institute of technology, Ube college, 
	      2-14-1, Tokiwadai, Ube, Yamaguchi, JAPAN 755-8555.}
\email{ykato@ube-k.ac.jp}

\newcommand{\A}{\mathbb{A}}

\newcommand{\Spec}{\mathrm{Spec}}

\newcommand{\Hom}{\mathrm{Hom}}

\newcommand{\sSet}{\mathrm{Set}_{\Delta}}

\newcommand{\Cat}[1]{\mathrm{Cat}_{#1}}

\newcommand{\Map}{\mathrm{Map}}
\newcommand{\Alg}{\mathrm{Alg}}
\newcommand{\Mod}{\mathrm{Mod}}
\newcommand{\AMod}{\mathrm{alMod}}
\newcommand{\CAlg}{\mathrm{CAlg}}

\newcommand{\tm}{\tilde{\mathfrak{m}}}
\newcommand{\Fun}{\mathrm{Fun}}

\newcommand{\FMod}{\mathrm{Mod}^{\rm fi}}
\newcommand{\CMod}{\mathrm{Mod}^{\rm cl}}
\newcommand{\APMod}{\mathrm{alPMod}}

\newcommand{\AMGL}{\mathrm{alMGL}}
\newcommand{\alFSyn}{\mathrm{alFSyn}}
\newcommand{\Tor}{\operatorname{Tor}}
\newcommand{\Ext}{\operatorname{Ext}}

\newcommand{\Sch}{\mathbf{Sch}}
\newcommand{\MS}{\mathbf{MS}}
\newcommand{\MSp}{\mathrm{MSp}}

\usepackage{enumerate}
\usepackage{array}
\usepackage[all]{xy} 
\usepackage{delarray}

\begin{document}
\begin{abstract}
Faltings~\cite{Faltings}; Gabber and Ramero~\cite{GR} introduced almost
mathematics. In another way, almost mathematics can be characterized
bilocalization abelian category of modules mentioned in Quillen's
unpublished note~\cite{Q}. Applying the concept of Quillen's
bilocalization to Gabber and Ramero's work, this paper establishes the
almost version of algebraic $K$-theory and cobordism. As a result of
almost $K$-theory, we prove that, in the case an almost algebra
containing a field, the almost $K$-theory of the almost algebra is a
direct factor of the $K$-theory of the field, implying that almost
$K$-theory holds the Gersten property. We clarify that an almost $K$-theory
is a $K$-theory spectrum of non-unital firm algebras in the sense of
Quillen~\cite{Q}.  Furthermore, we obtain that almost algebraic cobordism holds tilting equivalence on the category of zero-section stable integral perfectoid algebras with finite syntomic topology.
\end{abstract}
\maketitle
\section{Introduction}
\label{sec:introduction} Faltings~\cite{Faltings} first introduced
almost mathematics, proving almost purity.  More concisely, Gabber and
Ramero established~\cite{GR} almost ring theory in their
textbook:theories of almost modules, almost algebras, and almost
homotopical algebra.  While almost mathematics has various applications
to arithmetic geometry, for example, Scholze's work perfectoid
geometry~\cite{sh12}, Quillen~\cite{Q} mentioned linear algebra over
non-unital rings which is the same as almost mathematics.  Quillen's
work is more conceptional in the sense of using categorical language: in
his work, almost mathematics is characterized as bilocalization of an
abelian category of modules.

In almost mathematics, the theory of derived categories is a bit
complicated.  For instance, exact sequences in the almost-world are not
exact in the corresponding module category in general.  By considering
the derived category of a bilocalization of the category of modules,
this paper provides algebraic $K$-theory and cobordism of almost
mathematics. More preciously, we define the derived category (or stable
$\infty$-category) of almost perfect complexes, enabling us to obtain
the $K$-theory of almost modules. In this paper, we technically
introduce a new $K$-theory $K^+(A)$, for any almost algebra $A$, which
decomposes into almost $K$-theory of $A$ and the $K$-theory of almost
acyclic parts. Using the properties of $K^+(A)$, we prove the almost Gersten property (Theorem~\ref{AlmostGersten}) and an equivalence of almost $K$-theories between almost algebras and the corresponding non-unital algebras (Theorem~\ref{K-ideal}).

Furthermore, defining almost finite syntomic algebras, we can get the
defining algebraic cobordism in almost mathematics similar to Elmanto,
Hoyois, Khan, Sosnilo, and Yakerson's work~\cite{ModMGL} via framed
correspondences.

This paper's another result is the tilting equivalence between the
almost algebraic cobordism of zero-section stable integral perfectoid
algebras by a similar argument to the author's previous
work~\cite{K-MGL} about the algebraic cobordism of non-unital integral
perfectoid algebras (Theorem~\ref{limit-1}). This result lets us expect
some analogy between non-unital and almost algebra. Here the condition
zero-section stability is defined as follows: Let $i: \{0 \} \to
\Delta^\infty= \mathrm{colim}\Delta^n $ denote the canonical injection
and $\pi:\Delta^\infty \to \Delta^0 $ the projection. One has a Quillen
adjunction
\[ 
 \pi_* \circ i_* : (\sSet)_{/ \Delta^{0} } \rightleftarrows (\sSet)_{/ \Delta^{0} } :  \pi^* \circ i^*
\]
of simplicial model categories. A simplicial set $X$ is {\it
zero-section stable} if the morphism $\pi \circ i \circ \eta : X \to
\pi_*(i_*(X) )$ induced by the unit map $\eta: X \to ( i^* \circ \pi^*
)(( \pi_* \circ i_* ) (X)) $ is a weak equivalence.  Let $\A_S^\bullet$
denote the cosimplicial motivic space defined by
$\A_S^\bullet([n])=\A_S^n$ for each $n \ge 0$. Then the singular functor
$\Hom_S(\A^\bullet,\,- ): \MS \to \MS$ admits a left adjoint. We
We say that $X$ is locally zero-section stable if the simplicial sheaf
$\Hom_S (\A^\bullet_S,\,X )$ is stalk-wise zero-section stable. Let
$\MSp^0$ denote the full subcategory of zero-section motivic
spectra. Then the inclusion functor: $\MSp^0 \to \MSp$ admits a left
adjoint $\mathrm{Z}_0: \MSp \to \MSp^0$ which is defined to be the
homotopy colomoit
\[
  \mathrm{Z}_0(X)= \varinjlim_{n} \overbrace {(\pi_* \circ i_*)   \circ  (\pi_* \circ i_*)   \circ \cdots  \circ (\pi_* \circ i_*)}^n (X).   
\] 
We call $\mathrm{Z}_0(X)$ the {\it zero-section stabilization} of $X$.

This paper is organized as follows: In Section~\ref{sec:almost-derived},
we provide a brief review of Quillen's work~\cite{Q} bilocalization by
the Serre subcategory of almost zero modules. Section~\ref{sec:K-theory}
defines almost perfect complexes by using derived functors of
categorical equivalences between full subcategories spanned by local
objects and spanned by colocal objects. Via the equivalence, we prove
the comparison theorem of almost $K$-theories
(Theorem~\ref{TheoremAlmostK}). In Section~\ref{sec:Almost-MGL}, on the
viewpoint of non-unital algebra, we study the left adjoint of the
localization functor from the category of algebras to one of almost
algebras.  We define almost finite syntomic algebras, enabling us to
establish the almost version of algebraic cobordism by following
\cite{ModMGL}.  In the final part of this paper, localizing the stable
$\infty$-category of motivic spectra by zero-section stable finite
syntomic morphisms, we prove the tilting equivalence of the almost
algebraic cobordism for zero-section stable integral perfectoid
algebras.

\section{Bilocalization of abelian category of modules}
\label{sec:almost-derived}

In this paper, we fix an unital commutative ring $V$. In this section; we
review Quillen's unpublished note~\cite{Q} in the almost mathematics
language, giving the proof of the propositions.

Let $\mathfrak{m}$ be an idempotent ideal of $V$. A $V$-module
$M$ is said to be {\it almost zero} if $\mathfrak{m}M=0$. A
$V$-homomorphism $f:M \to N$ of $V$-modules is called {\it an almost
isomorphism} if both the kernel and the cokernel of $f$ are almost zero.

\begin{lemma}[\cite{Q}]
\label{I-tensor} Let $V$ be an unital commutative ring and $\mathfrak{m}$
an idempotent ideal. Then $\mathfrak{m} \otimes_V M=0$ if and only if
$M$ is almost zero.
\end{lemma}
\begin{proof}
The only if direction is clear. Assume $\mathfrak{m}M =0$. For $a
\otimes x \in \mathfrak{m} \otimes_V M \ (a \in \mathfrak{m},\,x \in
M)$, there exists $( a_i b_i) \subset \mathfrak{m}^2$ such that
$a=\sum_i a_i b_i $. Hence one has $ a \otimes x=(\sum_i a_i b_i
)\otimes x=\sum_i a_i b_i \otimes x=\sum_i a_i \otimes b_i x = 0$. \qed
\end{proof}

\begin{proposition}[\cite{Q}]
\label{almost} Let $V$ be an unital ring, $\mathfrak{m}$ an idempotent
ideal of $V$, and $M$ a $V$-module. Then the canonical morphisms $\mu:
\mathfrak{m} \otimes_V M \to M$ and $\mu': M \to
\Hom_{V}(\mathfrak{m},\,M)$ are both almost isomorphisms.
\end{proposition}
\begin{proof}
Since the abelian category $\Mod_V$ is both enough projective and
injective, the kernels and the cokernels;
$\mathrm{Tor}_i^V(V/\mathfrak{m},\,M )$ and
$\mathrm{Ext}_V^i(V/\mathfrak{m},\,M )$ for $i= 0,\,1$, are killed by
$\mathfrak{m}$. \qed
\end{proof}

\begin{definition}
\label{firm-closed} In the situation of Proposition~\ref{almost}, we say
that $M$ is {\it firm} if $\mu$ is an isomorphism and $M$ is {\it
closed} if $\mu'$ is an isomorphism.
\end{definition}
In particular, an injective $V$-module $Q$ is closed if and only if it is $\mathfrak{m}$-torsion free (i.e. $\Hom_V (V / \mathfrak{m},\,Q)=0$).

\begin{corollary}[\cite{Q}]
\label{I-tilde} Let $\mathfrak{m}$ be an idempotent ideal of an unital
ring $V$. Then the map $\mu_\mathfrak{m}:\mathfrak{m} \otimes_V
\mathfrak{m} \to \mathfrak{m}$ is an almost isomorphism and
$\mu_\mathfrak{m} \otimes \mathfrak{m} :\mathfrak{m}\otimes_V
\mathfrak{m} \otimes_V \mathfrak{m} \to \mathfrak{m} \otimes_V
\mathfrak{m}$ already an isomorphism.
\end{corollary}
\begin{proof}
This follows from Lemma~\ref{I-tensor} and Proposition~\ref{almost}. \qed 
\end{proof}

\begin{corollary}[\cite{Q} Proposition 4.1 and Proposition 5.3]
  Let $V$ be an unital ring and $\mathfrak{m}$ an idempotent ideal of
  $V$. Write $\tm= \mathfrak{m} \otimes_V \mathfrak{m}$.  For any
  $V$-module $M$, $\tm \otimes_V M$ is firm and $\Hom( \tm,\,M)$ closed.
\end{corollary}
\begin{proof}
By Corollary~\ref{I-tilde}, the counit $\mu_{ \tm \otimes_V M} :
\mathfrak{m} \otimes_V \tm \otimes_V M \to \tm \otimes_V M $ is an
isomorphism, inducing canonical isomorphisms $\Hom(\tm,\,M)\simeq
\Hom(\mathfrak{m} \otimes_V \tm ,\,M) \simeq
\Hom(\mathfrak{m},\,\Hom(\tm,\,M))$. \qed
\end{proof}

Let $M$ be a closed module and an injective hall $ M \to Q $. Then, by
the right exactness of $\Hom_V (\mathfrak{m},\,-)$, one has an injection
$M \to\Hom_V (\tm,\,Q)$. Then $\Ext_V^1(N,\, \Hom_V(\tm,\,Q
))=\Ext_V^1(N,\, \mathbb{R}\Hom_V(\tm,\,Q )) \simeq
\Ext_V^1(N\otimes^{\mathbb{L}}_V \tm ,\, Q )=0$ implies that
$\Hom_V(\tm,\,Q ) $ is injective. Hence, a $V$-module is closed if and
only if it is a kernel of a homomorphism of closed injective modules.

Let $\mathcal{S}$ denote the Serre subcategory of $\Mod_V$ spanned by
almost zero modules.  We say that a $V$-module $M$ is
$\mathcal{S}$-colocal (resp. $\mathcal{S}$-local) if for any almost
isomorphism $f:N_1 \to N_2$, the induced map
\[	 
  f_*:\Hom_V(M,\,N_1) \to \Hom_V(M,\,N_2)\, \quad (\text{resp. } f^*:\Hom_V( N_2,\,M) \to \Hom_V(N_1,\,M))						
\]
is an isomorphism.  The following well-known proposition is crucial for
almost mathematics.  We explain the proof:
\begin{proposition}[\cite{Q} Proposition 4.1 and Proposition 5.3]
\label{almost-film-closed} Let $V$ be an unital ring, $\mathfrak{m}$ an
idempotent ideal of $V$ and $M$ a $V$-module.  Let $\mathcal{S}$ denote
the Serre subcategory of almost zero modules of $\Mod_V$. Then the
following conditions are equivalent:
\begin{enumerate}[(1)]
 \item The $V$-module $M$ is $\mathcal{S}$-colocal.
 \item The $V$-module $M$ is firm. 
\item The $V$-module $M$ satisfies that
       $\Ext^0_V(M,\,N)=\Ext_V^1(M,\,N)= 0$ for any almost zero module
       $N$.
\end{enumerate}
Similarly, the following conditions are equivalent:
\begin{enumerate}[(1)']
\item The $V$-module $M$ is $\mathcal{S}$-local.
\item The $V$-module $M$ is closed.
\item The $V$-module $M$ satisfies that
       $\Ext^0_V(N,\,M)=\Ext_V^1(N,\,M)= 0$ for any almost zero module
       $N$.
\end{enumerate}
\end{proposition}
\begin{proof}
Assume that $M$ is $\mathcal{S}$-local. By Proposition~\ref{almost}, the
induced map
\[
    \mu_*: \Hom(M,\, \mathfrak{m}\otimes_V M  ) \to    \Hom(M,\, M  )
\]
is bijective. Hence $\mu: \mathfrak{m}\otimes_V M \to M$ is surjective,
in particular $M=\mathfrak{m}M$. Since $M $ is a direct factor of
$\mathfrak{m} \otimes_V M$, the composition $M \to \mathfrak{m}
\otimes_V M \to \tm \otimes_V M$ is split. A direct factor $M$ of the
firm module $\tm \otimes_V M$ is also firm.

Assume that $M$ is firm. Let $N$ be an almost zero module. We prove that
$ \Ext^i_V(M,\,N )=0$ for $i=0,\,1$. The condition $\mathfrak{m}N=0$
implies that $N$ is also $V/\mathfrak{m}$-module. The adjunction between
the derived categories:
\[
( V/\mathfrak{m}) \otimes_V^{ \mathbb{L}}- : D(V) \rightleftarrows    D(V/\mathfrak{m}): \mathbb{R} \Hom_{V} (V/\mathfrak{m},\,- )
\]
induces a canonical equivalence $\mathbb{R} \Hom_{V} ((V/\mathfrak{m})
\otimes_V^{\mathbb{L}} M,\,N) \simeq \mathbb{R} \Hom_{V} (M,\,N)$ and
the spectral sequence
\[
	E_{2}^{pq}= \Ext_{V/\mathfrak{m}}^{p} (\Tor_{q}^V(V/\mathfrak{m},\, M),\,N )
\Longrightarrow \Ext_V^{p+q}(M,\,N) \] where $E_2^{pq}=0$ for
$q=0,\,1$. Hence one has $\Ext_V^{0}(M,\,N)= \Ext_V^{1}(M,\,N)=0$.

Assume that condition (3). Let $f:N_1 \to N_2$ be a morphism of
$V$-modules with almost zero modules $N'= \mathrm{Ker} f$ and
$N''=\mathrm{Coker} f$. Write $N= \mathrm{Im} f$ and consider an exact
sequences:
\[
  0 \to N' \to N_1 \to N \to 0  \quad \text{and} \quad   0 \to N \to N_2 \to N'' \to 0.
\] 
We obtain that the induced map $f_*:\Hom_V(M,\,N_1) \to \Hom_V(M,\,N)
\to \Hom_V(M,\,N_2)$ is composition of isomorphisms by the assumption
(3).

Next, assume that $M$ is $\mathcal{S}$-colocal. Then the almost
isomorphism $:\mathfrak{m} \to V$ induces an isomorphism $M \to
\Hom_{V}(\mathfrak{m},\,M)$. Therefore $M$ is closed.

If $M$ is closed, there exist an injective resolution $0 \to M \to Q^0
\to Q^1 \to Q^2 \to \cdots $, where $Q^i$ is closed for $i=0,\,1$. By
the canonical isomorphism $\Hom_V(N ,\, \Hom_{V} (\mathfrak{m}, \,M))
\simeq \Hom_V(N\otimes_{V} \mathfrak{m},\, M )$, for any almost zero
module $N$, $\Hom_V(N,\,Q^i) \simeq \Hom_V(N\otimes_{V} \mathfrak{m},\,
Q^i )=0 \ (i=0,\,1)$. Therefore we obtain $\Ext_V^i(N,\,M)=0$ for
$i=0,\,1$.

By the similar argument of the above (3) to (1), the implication (3)' to
(1)' holds.

\qed
\end{proof}

\begin{theorem}[\cite{Q} Theorem 4.5 and Theorem 5.6]
\label{A} Let $V$ be an unital ring and $\mathfrak{m}$ an idempotent
ideal of $V$.  Let $\mathcal{S}$ denote the Serre subcategory of almost
zero modules of $\Mod_V$ and $\AMod_V$ the localization of $\Mod_V$ by $\mathcal{S}$. Then the localization functor $(-)^a: \Mod_{V}
\to \AMod_V$ is also a colocalization. Furthermore, the left adjoint is
isomorphic to the functor $\tm \otimes_V -: \AMod_V \to \Mod_V$ whose
essential image is the full subcategory of film modules, and the right
adjoint is isomorphic to $\Hom_V(\tm,\,-): \AMod_V \to \Mod_V$ whose
essential image is the full subcategory of closed modules. \qed
\end{theorem}

Let $V$ be an unital ring and $\mathfrak{m}$ an idempotent ideal of $V$.
An {\it almost $V$-module} is an object of $\AMod_V$ in Theorem~\ref{A} and an {\it almost $V$-algebra} a commutative algebra object of $\AMod_V$. 
Let $A$ be a unital $V$-algebra and $(-)^a:\Mod_A \to \AMod_A$ denote the
bilocalization of $\Mod_A$ by the Serre subcategory of almost zero
$A$-modules.

\begin{corollary}
Let $V$ be an unital ring and $\mathfrak{m}$ an idempotent ideal of $V$
and $A$ a unital $V$-algebra.  Let $\FMod_A$ denote the full subcategory of
$\Mod_A$ spanned by firm $A$-modules and $\CMod_A$ by closed
$A$-modules. Then those adjunctions
\begin{align*}
  \tm \otimes_V (-)   : \AMod_A &\rightleftarrows \FMod_A: (-)^a \\
                  (-)^a :\CMod_A &\rightleftarrows \AMod_A : \Hom_V(\tm ,\,-  )
 \end{align*}
are categorical equivalences. Assume that $\tm$ is a flat
$V$-module. The functor $\tm \otimes_V (-)$ induces categorical
equivalences of the derived categories:
\begin{align*}
    \tm \otimes_V (-):    D(\AMod_A) &\rightleftarrows D(\FMod_A  ): (-)^a \\
      \tm \otimes_V (-) :  D(\CMod_A) &\rightleftarrows \   D(\FMod_A): \mathbb{R} \Hom_V(\tm ,\,-  ).
\end{align*} \qed
\end{corollary}
Write $(-)_!= \tm \otimes_V \Hom_V (\tm,\,(-) ): \AMod_A \to
\FMod_A$. Since $\tm$ is flat, the functor $(-)_!$ is left
exact. Furthermore, if $f:M \to N$ is almost surjective, one has
$(\mathrm{Coker} f)_!=0$, being $(-)_!$ is right exact. Hence $(-)_!$ is
an exact functor.

\section{Almost algebraic $K$-theory}
\label{sec:K-theory}
\subsection{Almost perfect complexes}
In this section, we always assume that $\tm$ is a flat $V$-module.  

\begin{definition}
An $A$-module $M$ is {\it almost finitely generated} (resp. {\it almost
finitely presented}) if for any filtered inductive system $(N_\alpha)$
of firm $A$-module, the canonical map
\begin{equation}
\label{compact}
\varinjlim  \Hom_{\FMod_A}(\tm \otimes_V M,\,N_\alpha ) \to  
\Hom_{\FMod_A}(\tm \otimes_V M,\,\varinjlim N_\alpha )   
\end{equation}
is almost injective (resp. an almost isomorphism). 
\end{definition}

\begin{definition}
An $A$-module $P$ is {\it almost projective} if the functor
$\Hom_{\AMod_A}(P,\, - ) : \AMod_A \to \AMod_A$ is exact, equivalently,
$\tm \otimes_V   \Hom_{\FMod_A}(\tm \otimes_V P,\, - ): \FMod_A \to \FMod_A$ is exact.
\end{definition}
\begin{proposition}
Let $P$ be a finitely generated projective $A$-module. Then $\tilde{\mathfrak{m}} \otimes_V P$ is an almost finitely presented projective $A$-module.  
\end{proposition}
\begin{proof}
Since $P$ is a direct factor of finitely generated free $A$-module, it is sufficient to check that $\tm \otimes_V A$ is an almost finitely presented projective $A$-module. Then the source of the canonical morphism (\ref{compact}) is 
isomorphic to $\varinjlim (N_\alpha)_*$ and the target $(\varinjlim N_\alpha )_*$. Therefore,  (\ref{compact}) is an isomorphism after tensoring with $\tm$, implying that $\tm \otimes_V A $ is almost finitely presented. For any firm module $M$, both the morphisms of the diagram $M  \leftarrow  \tm \otimes_V M  \rightarrow  \tm \otimes_V (M_*)$ are isomorphisms. Hence  the functor $\Hom_{\FMod_A}(\tm \otimes_V A,\, - ) = \tm \otimes (-)_*$ is clearly exact. \qed   
\end{proof}

Let $\APMod_A$ denote the full subcategory of $\Mod_A$ spanned by those objects $\tilde{\mathfrak{m}} \otimes_V P$ where $P$ is a finitely generated projective $A$-module.

\begin{proposition}
\label{firm-compact}
The firm module $\tm \otimes_V A$ is a compact object of
$D(\FMod_A)$. \qed
\end{proposition}

We define the $K$-theory of almost mathematics by using
Thomason--Trobaugh's~\cite{TT} method: For any algebra $A$,
$\mathrm{Perf}(A)$ denotes the full triangulated subcategory of
$D(\Mod_A)$ generated by an object $A$, and an object of
$\mathrm{Perf}(A)$ is called a {\it perfect complex}. Any full
triangulated subcategory of $D(\Mod_A)$ generated by some objects is
assumed to be idempotent complete throughout this section.  We define an
almost version of perfect complexes:
\begin{definition}
Let $\mathrm{APerf}(A)$ denote the full triangulated subcategory of
$D(\FMod_A)$ generated by $\tm \otimes_V A$. An {\it almost perfect
complex} $E$ is an object of $D(\AMod_A)$ satisfying $ \tm \otimes_V E
\in \mathrm{APerf}(A)$. That is. The category $\mathrm{Perf}^{\rm
al}(A)$ of almost perfect $A$-complexes is defined to be the pullback $
D(\AMod_A) \times_{D(\FMod_A)} \mathrm{APerf}(A)$ of triangulated
categories.
\end{definition}

\begin{definition}
Let $A$ be an almost $V$-algebra and $K^{\rm al}(A)$ denote the
$K$-theory spectrum of the triangulated category of
$\mathrm{APerf}(A)$. We call $K^{\rm al}(A)$ the {\it almost
$K$-theory spectrum} of $A$.
\end{definition}

We will characterize almost $K$-theory.  
\begin{proposition}
\label{firm-perfect} The exact functor $\tm \otimes_V (-): \Mod_A \to
\FMod_A$ induces an essential surjective functor form the triangulated
category of perfect $A$-complexes to the one of almost perfect
$A$-complexes.  \qed
\end{proposition}
\begin{lemma}
\label{PMod-closed} Let $P$ be a finitely generated projective
$A_*$-module. Then $P$ is closed.
\end{lemma}
\begin{proof} 
Clearly $A^n_*$ is closed for any integer $n \ge 0$.  By the assumption,
$P$ is a direct factor of some $A_*^n$, implying that $P$ is
closed. \qed
\end{proof}
By Proposition~\ref{almost-film-closed}, $\tm \otimes_V(-) $ and $\Hom_V
( \tm,\,- )$ invert almost isomorphisms, implying that one has $ \tm
\otimes_V ( \Hom_V ( \tm,\,A ) ) \simeq \tm \otimes_V A $ and $\Hom_V (
\tm,\,\tm \otimes_V A ) \simeq \Hom_V ( \tm,\, A )$.


\begin{remark}
While the triangulated category $D(\mathrm{Perf}(A_*))$ is contained in
$D(\CMod_A)$, which is distinct to $D(\Mod_{A_*})$. An $A_*$-module $\tm
\otimes_V A_*$ is not closed in general. Indeed, if $A_*$ is an unital
ring, $\Hom_V(\tm,\,\tm \otimes_V A_* ) \simeq A_* \not\simeq \tm
\otimes_V A_* \not\ni 1$.
\end{remark}


\subsection{Localization by almost quasi-isomorphisms}
Let $f:E \to E'$ be a morphism of
$A$-complexes. Then $f$ is said to be an {\it almost quasi-isomorphism}
if the induced morphism $\tm \otimes_V f: \tm \otimes_V E \to \tm
\otimes_V E' $ is a quasi-isomorphism.
\begin{proposition}
\label{almost-qis} Let $f:E \to E'$ be a morphism of $A$-complexes. Then
 the following conditions are equivalent:
\begin{itemize}
 \item[(1)] The morphism $f$ is an almost quasi-isomorphism. 
 \item[(2)] The firm complex $ \tm \otimes_V \mathrm{cone} f$ is acyclic. 
\end{itemize}
\end{proposition}
\begin{proof}
The proof is straightforward. \qed 
\end{proof} 



Let $A$ be an almost $V$-algebra and $\mathrm{Perf}(A)$ denote the triangulated
category of perfect complexes of $A$-modules. Choosing classes of weak
equivalences on $\mathrm{Perf}(A)$ defines the following $K$-theory
spectra:
\begin{definition}
\label{DefTT}
We define $K$-theory spectra as the following:
\begin{itemize}
\item The almost Thomason--Trobaugh $K$-theory $K^{\rm aTT}(A)$ is the
      Waldhausen $K$-theory of $\mathrm{Perf}(A)$ whose
      weak-equivalences are almost quasi-isomorphisms.
\item The Thomason--Trobaugh $K$-theory $K^{\rm TT}(A^{\mathfrak{m}})$
      is the Waldhausen $K$-theory of the category
      $\mathrm{Perf}^{\mathfrak{m}}(A)$ spanned by almost acyclic
      perfect $A$-complexes.
\end{itemize} 
\end{definition}

To compare almost $K$-theories mentioned above, we recall the following the
approximation theorem:
\begin{theorem}[\cite{W1}, Theorem 1.6.7 (Approximation theorem)]
\label{W-app} Let $\mathcal{C}$ and $\mathcal{D}$ be small 
 Waldhausen categories whose weak equivalences satisfy the 2-out-of-3 property and $F: \mathcal{C} \to \mathcal{D}$ an
 exact functor.  If the following conditions hold:
\begin{itemize}
 \item[(1)] The category  $\mathcal{C}$ is closed under cylinder objects.   
 \item[(2)] A morphism in $f$ is a weak equivalence in $\mathcal{C}$ if
	    and only if $F(f)$ is a weak equivalence in $\mathcal{D}$.
 \item[(3)] The functor $F$ is a coCartesian fibration. That is, for any
	    object $c$ of $\mathcal{C}$ and any map $x: F(c) \to d$ in
	    $\mathcal{D}$, there exists a map $a: c \to c'$ and a weak
	    equivalence $x': F(c') \to d$ such that $x = x' \circ F(a)$.
\end{itemize} 
Then the induced morphism $K(F): K(\mathcal{C} ) \to K(\mathcal{D})$ is
a weak equivalence.  \qed
\end{theorem}

\begin{lemma}
\label{clynder} Those triangulated categories in Definition~\ref{DefTT}
 admit cylinder objects.  
\end{lemma}
\begin{proof}
Since $\mathfrak{m} \otimes_V -$ and $-\otimes_V \mathfrak{m}$ are exact
and commute with shift and the mapping cone functor, the all of
$\mathrm{Perf}(A),\,\mathrm{APerf}(A)$, and $\mathrm{Perf}^{\rm al}(A)$, admit the usual mapping cylinder complexes
$\mathrm{cyl}(f:E \to F)= \mathrm{cone}(g: \mathrm{cone}f[-1] \to E) $
satisfying the cylinder axioms. \qed
\end{proof}

The following theorem holds without the flatness of $\mathfrak{m}$:
\begin{theorem}
\label{TheoremAlmostK} Assume that $\tm$ is a flat $V$-module. Then the
functors $ \mathrm{Perf}(A_*) \to \mathrm{Perf}^{\rm al}(A) \to
\mathrm{APerf}(A)$ induces weak equivalences
\[
K(\tm \otimes_V (-)) : K^{\rm aTT}(A_*) \to  K(\mathrm{Perf}^{\rm al}(A)) 
 \to  K^{\rm al}(A) 
\]
of $K$-theory spectra. 
\end{theorem}
\begin{proof} 
Clearly, $\tm \otimes_V (-)$ is exact and sends all almost
quasi-isomorphisms to quasi-isomorphism of almost perfect complexes. By
definition of almost quasi-isomorphisms, the condition (2) in
Theorem~\ref{W-app} is satisfied. By the definition of
$\mathrm{Perf}^{\rm al}(A)$ and $\mathrm{APerf}(A)$, the condition (3)
in Theorem~\ref{W-app} holds, implying that the canonical morphism $
K(\mathrm{Perf}^{\rm al}(A)) \to K^{\rm al}(A)$ is an equivalence.  Let
$E$ be a perfect $A$-complex and $F$ be a firm almost perfect
complex. Given an morphism $f:\tm \otimes_V E \to F \simeq \tm \otimes_V
F $, we obtain a morphism $\Hom(\tm,\,f): \Hom_V(\tm,\,E ) \to
\Hom_V(\tm,\,F )$. Then $ \Hom_V(\tm,\,E )$ and $\Hom_V(\tm,\,F )$ are
perfect complexes of $A_*$-modules. Therefore, $ K^{\rm aTT}(A_*) \to
K^{\rm al}(A) $ is also an equivalence.
\qed
\end{proof}


\subsection{Splitting of almost $K$-theory}

In most cases, $\mathfrak{m} \otimes_V A$ or $A/\mathfrak{m}A$ is not
finitely presented over $A$, in particular, being not an object of
$\mathrm{Perf}(A)$. Let $\mathrm{Perf}^+(A)$
denote the
stable full subcategory of the triangulated category $D(\Mod_A)$
generated by $A$ and $\tm \otimes_V A$
and $K^{+}(A)$
denote the $K$-theory spectrum of
$\mathrm{Perf}^+(A)$.  
Then the functor $\tm \otimes_V (-): \mathrm{Perf}^+(A) \to \mathrm{Perf}^+(A)$ is categorically idempotent, inducing a homotopy projector $K(\tm \otimes_V
(-) ) : K^{+}(A) \to K^+(A)$. Furthermore, the essential image of $\tm \otimes_V (-): \mathrm{Perf}^+(A) \to \mathrm{Perf}^+(A)$ is equivalent
to $ \mathrm{APerf}(A)$. Thus, the canonical map $K^{\rm al}(A) \to
K^+(A)$ is homotopically split. We obtain the following:
\begin{theorem}
\label{split} Let $A$ be an almost $V$-algebra. Assume that $\tm$ is a flat $V$-module.  Then one has the following weak equivalences:
\[
  K^+(A) \simeq K^{\rm al}(A) \oplus K( \mathrm{Perf}^+(A)^\mathfrak{m}). 
\]
\qed 
\end{theorem}

\begin{corollary}
The projection $:K^+(A) \to  K( \mathrm{Perf}^+(A)^\mathfrak{m})$ is induced by 
the functor 
\[
 \varphi:   E \mapsto \mathrm{cone}(\tm \otimes_V E \to E  ). 
\]
\end{corollary}
\begin{proof}
For any complex $E$, one has $\tm \otimes \varphi(E) \simeq
\mathrm{cone}(\tm \otimes_V E \to \tm \otimes_V E)$, being acyclic.  If
$E$ is almost acyclic, one has a chain of quasi-isomorphisms $\varphi(E)
\simeq \mathrm{cone}(0 \to E ) \simeq E$. Therefore $\varphi$ induces a
homotopic identity on $K(\mathrm{Perf}^+(A)^\mathfrak{m})$ and one has
$\varphi(\varphi(E)) \simeq \varphi(E)$. Furthermore, $\varphi(\tm
\otimes_V E)$ are quasi-isomorphic to an acyclic complex
$\mathrm{cone}(\tm \otimes_V E \to \tm \otimes_V E )$, the functor
$(\tm\otimes_V (-)) \oplus \varphi$ induces a homotopic identity
map. \qed
\end{proof}

Using Theorem~\ref{split}, we can obtain the almost Gersten property of almost $K$-theory:
\begin{theorem}
\label{AlmostGersten}
Assume that there is a field $F$ containing $V$. For any almost $V$-algebra $A$ and $A \otimes_V F$-algebra $C$, let $\overline{K^\mathfrak{m}}(C)$ denote the cofiber of the composition $K^{\rm al}(A) \to K^{+}(A) \to K(C)$ induced by  the functor $(-)\otimes_A C: \Mod_A \to \Mod_{C}$. Then we have a homotopical decomposition  \[
K(C) \simeq K^{\rm al}(A) \amalg \overline{K^\mathfrak{m}}(C) 
\]
of $K$-theory spectra. In particular, the pullback $K(-\otimes_V F):K^{\rm al}(A) \to K(A \otimes_V F)$ is homotopically split injective.  
\end{theorem}
\begin{proof}
Since the canonical morphism $\tm \otimes_V C \to C$ is an isomorphism,  the $\mathrm{Perf}(C) \to \mathrm{Perf}^+(C)$ functor is a categorical equivalence of triangulated categories. By definition of $\overline{K^\mathfrak{m}}(C)$, the $K$-theory spectrum $\overline{K^\mathfrak{m}}(C)$ is equivalent to the homotopy coCartesian product 
$K(C) \amalg_{K^+(A)} K^\mathfrak{m}(A)$. Therefore, by the formal argument, we have a chain of equivalences:
\begin{multline*}
 \overline{K^\mathfrak{m}}(C) \amalg K^{\rm al}(A) \simeq \left( K(C) \amalg_{K^+(A)} K^\mathfrak{m}(A) \right) \amalg K^{\rm al}(A) \\  \simeq   K(C) \amalg_{K^+(A)} \left(    K^\mathfrak{m}(A)  \amalg K^{\rm al}(A)    \right) \simeq  K(C) \amalg_{K^+(A)} K^+(A) \simeq K(C). 
\end{multline*} 
 \qed
\end{proof}

\subsection{The left adjoint of the localization $(-)^a: \CAlg(\Mod_A) \to \CAlg(\AMod_A)$ from  the non-unital algebra viewpoint}

We will review almost mathematics from the theory of non-unital algebra. 
We assume that $1 \not \in\mathfrak{m}$. Let $A$ be an almost
$V$-algebra and $B$ an $A$-algebra. Then $B_! :=\tm \otimes_V
\Hom_A(\tm,\,B)$ is a (non-unital) $B$-algebra. Consider the
unitalization functor
\[
V \oplus (-): \CAlg^{\rm nu}(\Mod_A) \to \CAlg(\Mod_A)
\]
Then the direct sum $V \oplus B_!$ has a
canonical unital ring structure defined by
\[
 (v,\,b) \cdot (v',\,b') = (vv',\, vb'+v'b+bb')  
\]
for $v, \, v' \in V$ and $b,\,b' \in B$. Note the functor $V \oplus
-:\CAlg^{\rm nu}(\Mod_A) \to \CAlg(\Mod_A) $ induces a categorical
equivalence between $\CAlg^{\rm nu}(\Mod_A)$ and the category of
augmented commutative rings $\CAlg(\Mod_A)_{/ V}$.  The right adjoint is
the augmented ideal functor $\mathrm{Ker} ( - \to V):\CAlg(\Mod_A)_{/ V}
\to\CAlg^{\rm nu}(\Mod_A)$, becoming the quasi-inverse of $V \oplus - $.

Let $B_{!!}$ denote the coCartesian product $V \oplus_{\tm} B_{!} $. Then one has an exact sequence
\[
 \tm \to V  \oplus   B_! \to B_{!!} \to 0,
\]
where the left map is almost injective.  After tensoring with $\tm$,
indeed, one has a trivial split exact sequence:
\[
   0 \to \tm \to \tm  \oplus  (\tm \otimes_V  B_!) \to \tm 
\otimes_{V} B_{!!} \to 0,
\]
implying that $\tm \otimes_V B_! \to \tm \otimes_{V} B_{!!}$ is an
isomorphism. Therefore the functor $(-)_{!!}:
\CAlg(\Mod_A) \to \CAlg(\AMod_{A_{!!}})$ is left adjoint to the
localization $(-)^a: \CAlg(\Mod_{A_{!!}}) \to \CAlg(\AMod_A)$.
If $\tm  \to V\oplus (\tm \otimes_V B_!)$ is exactly injective, $B$ is
called an {\it exact} almost $V$-algebra~\cite[Definition
2.2.27]{GR}. In the case $\tm \simeq \mathfrak{m}$, in particular, that
$\mathfrak{m}$ is flat, $B$ is always exact~\cite[Remark 2.2.28]{GR}.

\begin{remark}
 In Hovey's Smith ideal theory~\cite{Smith-ideals} language, a $V$-homomotphism $j:\tm \to V$ is regarded a Smith ideal of $V$. For any $A$-algebra $B$, the push-forward product 
\[
j    \Box (\eta_B :V \to B) = (V  \oplus_{\tm } ({\tm \otimes B} )      ) \to B
\]
is the canonical ring homomorphism $B_{!!} \to B$. Since $j$ induces an isomorphism $:\mathrm{id}_{\tm} \Box j \simeq \mathrm{id}_{\tm}$, the induced homomorphism
\[
 (\tm \otimes_{V} {B}_{!!} \to \tm \otimes_{V} B  ) =       \mathrm{id}_{\tm } \Box j \Box (V \to B) \simeq  \mathrm{id}_{\tm } \Box ( V \to B)=(\tm \otimes_{V} B \to \tm \otimes_{V} B  )
\]
is an isomorphism, 
entailing that the canonical ring homomorphism $B_{!!} \to B$ is an almost isomorphism of $A$-algebras. 
\end{remark}

\begin{lemma}
\label{A!!-projective} Let $E$ be an almost perfect $A_{!!}$-complex. Then $E$ is also an almost perfect $V \oplus \tm \otimes_V A$-complex.
\end{lemma}
\begin{proof}
Write $B=V \oplus \tm \otimes_V A$. Then $\tm \otimes_V B= \tm \oplus \tm \otimes_V A$. Therefore $  \tm \otimes_V A \simeq \tm \otimes_V A_{!!}$ is a direct factor of $\tm \otimes_V B$. Since $\mathrm{APerf}(B)$ is closed under retracts, $\tm \otimes_V A_{!!}$ is an object of $\mathrm{APerf}(B)$.     
 \qed
\end{proof} 
\begin{corollary}
\label{KB}
For any almost $V$-algebra $A$, one has an decomposition: 
\[
 K^{\rm al}(V \oplus \tm \otimes_V A ) \simeq K^{\rm al}(V) \oplus K^{\rm al}(A).
\] 
\qed 
\end{corollary}

Write $B=V \oplus \tm \otimes_V A$. 
Let $K^+(\tm \otimes_V A)$ denote the homotopy fiber of the map $K^+(-\otimes_B   V:K^+(B) \to K^+(V)$.   For any $V$-complex $E$, one has an isomorphism $E \to E \otimes_V B \otimes_B V $ induced by the identity composition $V \to B \to V$. Therefore the canonical map $K^+(B) \to K^+(V)$ has a homotopically section. 
Therefore we have a homotopically decomposition:
\[
 K^+( B) \simeq K^+(V) \oplus K^+(\tm \otimes_V A).
\]

By Corollary~\ref{KB}, one has a weak equivalence $K^{\rm al}(\tm
\otimes_V A) \simeq K^{\rm al}(A)$. Let
$\mathrm{Perf}^+(B,\, \tm \otimes_V A)$ denote the full subcategory of $\mathrm{Perf}(B)$ spanned by $V$-acyclic complexes.  That is. A $B$-complex $E$ is $V$-acyclic if $E \otimes_B V  $ is acyclic. Note that for
any $B$-complex $E$, $E \otimes_B (\tm \otimes_V A)$ is already firm, implying any $V$-acyclic complex is quasi-isomorphic to a firm complex. Therefore one has $K^+(B,\,\tm \otimes_V A)= K^{\rm al}(B,\,\tm \otimes_V A )$.   Furthermore, one has a canonical isomorphism: $E\otimes_B (\tm \otimes_V A) \simeq  \tm \otimes_V (E\otimes_B A_{!!})$, implying there is a canonical functor $(-) \otimes_B  A_{!!} \otimes_V \tm   :   \mathrm{Perf}^{+}(B,\, \tm \otimes_V A) \to \mathrm{APerf}(A_{!!})$. 
Further, we call a morphism $f:E \to E'$ of $B$-complexes {\it $-\otimes_{B}V$-quasi-isomorphism} if the induced morphism $f\otimes_B V:E \otimes_B V \to E' \otimes_B V$ is a quasi-isomorphism. 


The main result of almost $K$-theory is that the almost $K$-theory of $A$ is represented by the $K^+$-theory of the non-unital algebra $\tm \otimes_V A$:
\begin{theorem}
\label{K-ideal}
Let $A$ be an almost $V$-algebra. Let $K^+(\tm \otimes_V A)$ denote the homotopy fiber of the canonical map $ K^+(V \oplus \tm \otimes_V A) \to K^+(V)$. Then the $K$-theory spectrum $K^+(\tm \otimes_V A)$ is canonically equivalent to the almost $K$-theory $K^{\rm al}(A)$. 
\end{theorem}
\begin{proof}
Let $\mathrm{Perf}^+(B,\,V)$ denote the triangulated category whose objects are the same of $\mathrm{Perf}^+(B)$ and weak equivalences are $-\otimes_{B}V$-quasi-isomorphisms. 
By applying Theorem~\ref{W-app} to the exact functor $ -\otimes_{B} V: \mathrm{Perf}^+(V) \to  \mathrm{Perf}^+(B,\,V) $, the induced map $:K^+(B,\,V) \to K^+(V)$ is a weak equivalence, entailing that the induced map: $ K^+(\tm \otimes_{V} A)\to   K^{+}(B,\,\tm \otimes_V A  )$ is also a weak equivalence: 
For any object $E'$ of $\mathrm{Perf}^+(V)$, indeed, the extension $E' \otimes_V B$ contains $E'$ as a direct factor. Therefore, $E'$ is also an object of $\mathrm{Perf}^+(B)$ and the inclusion $E' \to E' \otimes_V B \simeq E' \oplus (E' \otimes_{V} \tm \otimes_V A)$ a $-\otimes_{B}V $-quasi-isomorphism.

Furthermore, for any $B$-complex, one has an equivalence  $ E \simeq (E \otimes_B V) \oplus (E \otimes_{B} ( \tm \otimes_{V } A))$, implying that  
any $V$-acyclic complex is quasi-isomorphic to a firm complexes. Hence, we obtain a chain of equivalences: $  K^+(\tm \otimes_{V} A)   \simeq K^+(B,\, \tm \otimes_{V} A) \simeq  K^{\rm al }(B,\,\tm \otimes_V A) \simeq K^{\rm al }(\tm \otimes_V A)$. By Corollary~\ref{KB}, we get a weak equivalence: $K^{\rm al}(A) \simeq K^+(\tm \otimes_{V} A)$.
\qed
\end{proof}

\section{Finite syntomic algebraic cobordism of almost algebras}
\label{sec:Almost-MGL} Fix a base ring $V$ which is unital and
commutative and an almost $V$-algebra $A$.  Let $\mathfrak{m}$ be an
idempotent ideal of $V$. Following the previous section, this section assumes that $\tm$ is a flat $V$-module.

\subsection{Definition of finite syntomic motivic model structure of the category of simplicial presheaves}  
\label{sec:MSS} Let $\mathcal{X}$ be a Grothendieck site with an interval
object $I$. We assume that $\mathcal{X}$ has enough points: That is, a
morphism $f: X \to Y$ in $\mathcal{X}$ is an isomorphism if $f_x: X_x
\to Y_x$ is an isomorphism of sets for any point $x:* \to \mathcal{X}$
where the functor $(-)_x: \mathcal{X} \to \mathit{Sets}$ denotes the
right adjoint of the induced functor $x_*: \mathit{Sets} \to
\mathcal{X}$.  A simplicial object $U_\bullet : \Delta^{\rm op } \to
\mathcal{X}$ with an augmentation $\pi: U_\bullet \to X \in \mathcal{X}$
is a {\it hypercover} of $X$ if the following conditions are hold:
\begin{itemize}
\item For any $n \ge 0$, $U_\bullet ([n])$ is a
coproduct of compact objects represented by small objects of $\mathcal{X}$. 
\item The augmentation $\pi:
U_\bullet \to X$ is a stalk-wise trivial Kan fibration: That is, $\pi_x: U_{x,\,\bullet} \to *$ is a trivial Kan fibration for any point $x : *
\to X$.
\end{itemize}

The category $\sSet$ of simplicial set has a proper combinatorial
simplicial model structure, called {\it Kan--Quillen model
structure}.  Then the injective model structure of the category
$\sSet^{\mathcal{X}^{\rm op}}$ of simplicial presheaves on $\mathcal{X}$
is also proper combinatorial.  Let $\MS^\Delta_\mathcal{X}$ denote the
Bousfield localization of $\sSet^{\mathcal{X}^{\rm op}}$ defined as
follows: A simplicial presheaf $F$ is {\it motivic local} if $F$
satisfies the following conditions:
\begin{itemize}
 \item The presheaf $F$ is stalk-wise fibrant.  
 \item For any hypercover $\pi: U_\bullet \to X$ of $X \in
\mathcal{X}$, the induced
map $F(f): F( X) \to F(|U_\bullet | )$ is a weak equivalence. Here the
       functor $ |-|: \Fun( \Delta^{\rm op},\,
\mathcal{X}) \to \mathcal{X}$ denotes the geometric realization of
simplicial objects.
\item The canonical
map $\mathbf{1}: I \to *$ induces a weak equivalence $F(U ) \to F(U
\times I)$ for any $U \in \mathcal{X}$.
\end{itemize}
A map $f: F \to G$ of simplicial presheaves on
$\mathcal{X}$ is a {\it motivic equivalence} if the induced
map
\[ 
f^*:\Hom
(G,\,Z) \to \Hom  (F,\,Z)
\]   
is a weak homotopy equivalence of simplicial sets for each motivic local
presheaf $Z$.  By \cite[p.56, Corollary 4.55]{MR2771591}, the injective
model structure of $\sSet^{\mathcal{X}^{\rm op}}$ is proper
combinatorial and symmetric monoidal. Therefore the Bousfield
localization $\MS^\Delta_\mathcal{X}$ of
$\mathbf{sSet}^{\mathcal{X}^{\rm op}}$ is also proper.

We apply the finite syntomic site on the category $\Sch^{\rm fp}_{V}$ of finitely presented schemes over $\Spec{V}$, and the
interval object $\A^1_V$.  We let $\MS_{\rm FSyn}$ denote the
$\infty$-category determined by the simplicial model category $
\MS^\Delta_{\Sch^{\rm fp}_{V}}$ and we call objects of $\MS_{\rm FSyn}$
{\it (finite syntomic) motivic spaces}.  Moreover, $\MSp_{\rm FSyn}$
denotes the stable $\infty$-category of (finite syntomic) motivic
spectra defined by
\[
 \MSp_{\rm FSyn}=  \varprojlim ( \cdots \overset{\Omega_{\mathbb{P}_+^1}}{\to} (\MS_{\rm FSyn})_*  \overset{\Omega_{\mathbb{P}_+^1}}{\to}\ (\MS_{\rm FSyn})_*),
\] 
where $(\MS_{\rm FSyn})_*$ the $\infty$-category of pointed motivic
spaces and $\Omega_{\mathbb{P}_+^1}(- ) = \Map_{(\MS_{\rm
FSyn})_*}(\mathbb{P}^1_+,\, -)$ is the pointed $\mathbb{P}^1$-loop
functor. The stable $\infty$-category $\MSp_{\rm FSyn}$ is the full
subcategory of the model category $\mathrm{Spt}_{\mathbb{P}^1_+}(
\MS^\Delta_{\mathrm{FSyn}})$ spanned by
$\mathbb{P}_+^1$-stable fibrant objects. (See \cite[Section
10.2]{Jardine2}).

  

\begin{theorem}[\cite{K-MGL} Corollary 3.9]
 The localization $:L_{\rm FSyn}: \MSp \to \MSp$ by the family of
 zero-section stable finite syntomic surjective morphisms induces a categorical equivalence of stable $\infty$-categories $:L_{\rm FSyn}(\MSp) \to \MSp^{\rm
 FSyn}$.  \qed
\end{theorem}


\subsection{Some of almost homotopical algebra}
\label{sec-almostDerivedNakayama}

We will recall and prove some result of almost homotopical algebra. 
\begin{definition}
Let $A$ be an almost $V$-algebra. The {\it Jacobson radical} is defined
to be $\mathrm{Jac}(A)=\mathrm{Jac}(A_*)^a \subset A$.
\end{definition}
Clearly $I \to \mathrm{Jac}(A)$ is almost injective if and only if $ \tm
\otimes_V I \subset \mathrm{Jac}(A_*)$.

\begin{definition}
Let $A$ be an almost $V$-algebra and $I$ an ideal of $A$. We say that $I$ is
{\it tight} if there exists a finitely generated ideal $\mathfrak{m}_0
\subset \mathfrak{m}$ and an integer $n \ge 0$ such that $I^n \subset
\mathfrak{m}_0 A$.
\end{definition}

\begin{proposition}[\cite{GR} Corollary 5.1.17.]
 Let $A$ be  an almost $V$-algebra, $I$ a tight ideal of $A$ contained
 in the Jacobson radical of $A$.  Then $I_*$ is contained in the Jacobson
 radical of $A_*$.\qed
\end{proposition}

\begin{lemma}[\cite{GR} Lemma 5.1.7.]
\label{almostNakayama} Let $A$ be an almost $V$-algebra, $I$ a tight
ideal of $A$ which is contained in the Jacobson radical of $A$. If $M$
is an almost finitely generated $A$-module satisfying $IM=M$, then
exactly $M=0$. \qed
\end{lemma}
\begin{corollary}[\cite{GR} Corollary 5.1.8.]
\label{almostlift} Let $A$ be an almost $V$-algebra, $I$ a tight ideal
of $A$ which is contained in the Jacobson radical of $A$. Let $f:M \to
N$ be a homomorphism of almost finitely generated projective
$A$-modules. If $f \otimes_A A/I: M/IM \to N/IN$ is an isomorphism. Then
$f$ is also an isomorphism.  \qed
\end{corollary}

\begin{theorem}[\cite{GR} Theorem 5.3.24.]
Let $A$ be an almost $V$-algebra, $I$ a tight ideal of $A$ which is
contained in the Jacobson radical of $A$. Write $A_n= A/I^n$ for each $n
\ge 1$, and $A_\infty=\varprojlim{A_n}$. Then
\[
  \varprojlim : 2-\varprojlim \APMod_{A_n} \to \APMod_{A_\infty}  
\]
is a categorical equivalence. \qed
\end{theorem}

\begin{proof}
Let $(P_n)$ be an inverse system of finitely generated projective
 $A_n$-modules. Since the canonical map $:A^r_n \to (\varprojlim A^r_n)
 \otimes_{A_\infty} A_n $ is an isomorphism for each $n \ge 1$ and $r
 \ge 0$, $P_n \to (\varprojlim P_n) \otimes_{A_\infty} A_n $ is an
 isomorphism for each $n \ge 1$. Then $(P_n) \to ((\varprojlim P_n)
 \otimes_{A_\infty} A_n )$ is an isomorphism of inverse systems. Then
 one has a chain of isomorphisms $ \Hom_V ( \tilde{\mathfrak{m}},\, P_n
 ) \simeq \Hom_V ( \tilde{\mathfrak{m}},\, (\varprojlim P_n)
 \otimes_{A_\infty}A_n ) \simeq \Hom_V ( \tilde{\mathfrak{m}},\,
 \Hom_{A_\infty} ((\varprojlim P_n)^\vee,\,A_n ) \simeq \Hom_V (
 \tilde{\mathfrak{m}} \otimes_{A_\infty}(\varprojlim P_n)^\vee,\,A_n )
 \simeq \Hom_V ( (\varprojlim P_n)^\vee,\, \Hom_V
 (\mathfrak{m},\,A_n))$.  Therefore $ \Hom_V ( \tilde{\mathfrak{m}},\,
 P_n ) \simeq \Hom_V ( \tilde{\mathfrak{m}},\, (\varprojlim P_n)
 \otimes_V A_n )$.

Conversely, let $P_\infty$ be a finitely generated projective $A_\infty$-module. Then we show that $ P_\infty \to \varprojlim P_\infty \otimes_{A_\infty} A_n $ is an isomorphism. Since $P_\infty$ is canonically isomorphic to the second dual $P_\infty^{\vee \vee}$, $P_\infty \otimes_{A_\infty} A_n \simeq \Hom_{A_\infty}( P^\vee_\infty,\,A_n)$. Therefore one has $\varprojlim P_\infty \otimes_{A_\infty} A_n \simeq\varprojlim \Hom_{A_\infty}( P^\vee_\infty,\,A_n) \simeq \Hom_{A_\infty}( P^\vee_\infty,\,A_\infty) \simeq P_\infty$. We obtain 
\[
 \Hom_V( \tilde{\mathfrak{m}},\,P_\infty ) \simeq \Hom_V( \tilde{\mathfrak{m}},\, \varprojlim P_\infty \otimes_{A_\infty} A_n ) \simeq  \varprojlim \Hom_V( \tilde{\mathfrak{m}},\,  P_\infty \otimes_{A_\infty} A_n ) 
\]
\qed
\end{proof}

\begin{theorem}
\label{lim-perfect} Let $A$ be an almost $V$-algebra and $I$ a tight ideal
which is contained in the Jacobson radical of $A$. For $n \ge 1$, set
$A_n =A/I^n$, and $A_\infty=\varprojlim A_n$. Then the canonical
adjunction $( - \otimes_A^{\mathbb{L}}A_n ) :D(A_\infty)
\rightleftarrows 2-\varprojlim D(A_n) : \mathbb{R}\varprojlim_n$ induces
a categorical equivalence
\[
 \mathbb{R}\varprojlim_n: 2-\varprojlim \mathrm{APerf}( {A_n} )  \to \mathrm{APerf}(A_\infty) 
\]
between stable $\infty$-categories. 
\end{theorem}
\begin{proof}
Let $E$ be an almost perfect $A_\infty$-complex. Since $E$ is a
dualizable object of the derived category $D(A^a_\infty)$, $E
\otimes^\mathbb{L}_{A_\infty} (- ) $ preserves all small
limits. Therefore the canonical morphism $:E_* \to \mathbb{R}
\varprojlim (E \otimes^\mathbb{L}_{A_\infty} A_n)_* $ is a
quasi-isomorphism.

Conversely, let $(E_n)$ be an inverse system of complex of almost
finitely generated projective $A_n$-modules. Since $(A_{*})_\infty \to
A_{*}/(I_*)^n$ is surjective for each $n \ge 1$, $\varprojlim E_n \to
E_n$ is also surjective for each $n \ge 1$. Therefore, clearly, the
canonical morphism $ (\varprojlim E_n) \otimes_{A_\infty} A_n \to E_n $
is a quasi-isomorphism for each $n \ge 1$. Hence the morphism
$((\varprojlim E_n) \otimes_{A_\infty} A_n) \to (E_n)$ is an inverse
system of quasi-isomorphisms. \qed
\end{proof}

\subsection{Almost finite syntomic morphisms}
We consider the restriction of the functor $\tm \otimes_V (-): \Mod_{V}
\to \FMod_{V}$ to the subcategory $\CAlg_V$ of commutative unital
$V$-algebras. Since the functor $\tm \otimes_V (-)$ is (strong)
symmetric monoidal, the essential image of the restriction $\tm
\otimes_V (-)|_{\CAlg_{V}}$ is equivalent to the full subcategory of the
category of non-unital commutative $V$-algebras $\CAlg_{V}^{\rm
nu}$. Let $\CAlg_{\tm}$ denote the full subcategory $\CAlg_{V}^{\rm nu}$
spanned by those objects of the essential image of $\tm \otimes_V
(-)|_{\CAlg_{V}}$.

\begin{theorem}
\label{almost-monoidal}
The functor $\tm \otimes_V (-):\CAlg_V  \to\CAlg_{V}^{\rm nu}  $ induces a categorical equivalence
\[
 \tm \otimes_V (-): \CAlg^{\rm al}_{V} \to \CAlg_{\tm}
\]
whose quasi-inverse is equivalent to the functor 
$(V \oplus_{\tm} (-))^a: \CAlg_{\tm} \to \CAlg^{\rm al}_V$, where $V \oplus_{\tm} (-)$ denotes the cokernel of the diagonal: $\tm \to V \oplus (\tm \oplus_{V} (-))$. 
\end{theorem}
\begin{proof}
Let $\mu_{A}: A \otimes_V A \to A$ and $\mu_{\tm}: \tm \otimes_{V} \tm \to \tm$ denote the multiplication maps. By the idempotentness of $\mathfrak{m}$,  the map $\mu_{\tm}: \tm \otimes_{V} \tm \to \tm$ is an isomorphism. Therefore, in the commutative diagram 
\[
 \xymatrix@1{ 
\tm \otimes_{V } A \otimes_{V } \tm \otimes_{V } A   
\ar[r]^{\mu_{\tm}} 
\ar[d]_{\mu_A} &  \tm \otimes_{V } A \otimes_{V } A \ar[d]^{\mu_A}
\\
\tm \otimes_{V} \tm \otimes_{V } A  \ar[r]_{\mu_{\tm}} 
& \tm \otimes_{V} A,  
 }
\]
both horizontal maps are isomorphisms.  We prove that the $\tm \otimes_V (-):\CAlg_{\tm} \to \CAlg^{\rm al}_V $ functor is fully faithful: By definition of the functor $(-)_{!!}: \CAlg^{\rm al}_V \to  \CAlg_V $,  the counit $ (V \oplus_{\tm}  (\tm \otimes_V (-)))^a \to \mathrm{Id}_{ \CAlg^{\rm al}_V }  $ is an equivalence of functors. Therefore $\tm \otimes_V (-)$ is fully faithful, entailing a categorical equivalence. \qed 
\end{proof}

By Theorem~\ref{almost-monoidal}, we can formulate theories of almost
algebras from the non-unital algebras viewpoint.  Let $A$ be an almost
$V$-algebra and $B$ an almost $A$-algebra.  Then $\tm \otimes_{V }B$ is
an object of $(\CAlg_{\tm})_{\tm \otimes_{V} A /}$.  Let $P_A$ be a
property of $A$-algebras and $\mathcal{P}_A$ denote the subcategory of
$\CAlg_A$ spanned by $A$-algebras satisfying the property $P_A$. Then
$\mathcal{P}_{\tm \otimes_{V} A}$ denote the subcategory of
$(\CAlg_{\tm})_{\tm \otimes_{V} A /}$ spanned by those objects whose
image of $(V \oplus_{\tm} (-)):(\CAlg_{\tm})_{\tm \otimes_{V} A /} \to
\CAlg_V$ belongs to $\mathcal{P}_A$. We say that $\mathcal{P}_{\tm
\otimes_{V} A}$ is the subcategory of $(\CAlg_{\tm})_{\tm \otimes_{V} A
/}$ spanned by almost $A$-algebras satisfying the property $P_A$. 

\begin{definition}
Let $A$ be an almost $V$-algebra and $B$ an almost $A$-algebra.
The almost relative cotangent complex $L_{B/A}$ is defined by the following:  
\[
L_{B/A}= B_{!!} \otimes_{(V^a \times B)_{!!}} L_{(V^a \times B)_{!!}/ (V^a \times A)_{!!}} 
\]
in the derived category $D(\Mod_{B_{!!}})$. 
\end{definition}
By Gabber--Ramero~\cite[Proposition 2.5.43 and Remark 2.2.28]{GR}, the canonical morphism $\tm \otimes_{V} L_{B/A} \to L_{B/A} $ is a quasi-isomorphism. Therefore, $ L_{B/A} $ can be an object of $D(\FMod_{B_{!!}})$. 

\begin{proposition}
\label{almost-cotangent}
Let $A$ be an almost $V$-algebra and $B$ an almost $A$-algebra. Then the canonical morphism $L_{B/A} \to L_{B_{!!}/A_{!!}}$ induced by the projection $V \times B \to B$ is an almost quasi-isomorphism. 
\end{proposition}
\begin{proof}
Since the projection $(V \times A)_{!!} \to A_{!!}$ is flat and surjective, by Stacks project~\cite[08QY, Lemma 91.8.4]{stacks-project}, the relative cotangent complex $L_{A_{!!}/ (V^a \times A)_{!!}}$ is acyclic.   Note that $B \simeq (V^a \times B)/ V^a \simeq  (V^a\times B) \otimes_{V^a \times A} \left( (V^a \times A) / V^a \right)$. Therefore the colimit preserving functor $L_{(-)_{!!}/{(V^a \times A)_{!!}}}$ induces a quasi-isomorphism $L_{( V^a \times B)_{!!}/{(V^a \times A)_{!!}}} \to    L_{B_{!!}/{(V^a \times A)_{!!}}}$. By the distinguish triangle 
\[
B_{!!} \otimes^\mathbb{L}_{( V^a \times A)_{!!}} L_{A_{!!}/{(V^a \times A)_{!!}}} \to L_{B_{!!}/{(V^a \times A)_{!!}}} \to L_{B_{!!}/A_{!!}} \to B_{!!} \otimes^\mathbb{L}_{( V^a \times A)_{!!}} L_{A_{!!}/{(V^a \times A)_{!!}}}[1]
\]
induced by the sequence $ ( V^a \times A)_{!!} \to A_{!!} \to B_{!!}$ of $V$-algebras, we obtain that the canonical morphism $L_{B_{!!}/{(V^a \times A)_{!!}}} \to L_{B_{!!}/A_{!!}} $ is a quasi-isomorphism. Hence, both $\tm \otimes_V L_{B_{!!}/A_{!!}} $ and $L_{B/A}$ are quasi-isomorphic to the same object of $D(\FMod_B)$. \qed  
\end{proof}

\begin{definition}
A morphism $f:A \to B$ of almost $V$-algebras is {\it almost finite syntomic}
if the following conditions are satisfied:
\begin{itemize}
 \item[(1)] An $A$-algebra $B$ is an almost finitely projective
	    $A$-module.
 \item[(2)] The almost relative cotangent complex $L_{B/A}$ is almost perfect and tor-amplitude in $[-1,\,0]$ in the category  $D(\Mod_{B_{!!}})$.
\end{itemize}
\end{definition}

%
\begin{theorem}
\label{nomumital-almostFSyn}
Let $\Alg_{\tm \otimes_{V} A}^{\rm FSyn}$ denote the full subcategory of $\CAlg_{\tm \otimes_{V} A} $ spanned by those objects $\tm \otimes_{V} B$ satisfying $(V \oplus_{\tm} \tm \otimes_{V} B)^a$ is an almost finite syntomic $(V \oplus_{\tm} \tm \otimes_{V} A)^a $-algebra. Then the categorical equivalence $(V \oplus_{\tm} \tm \otimes_{V} (-))^a: \CAlg_{\tm \otimes_{V} A} \to \CAlg_{A}^a $ induces a categorical equivalence between $\Alg_{\tm \otimes_{V} A}^{\rm FSyn}$ and the full subcategory $\Alg^{\rm alFSyn}_A$ of $ \CAlg_{A}^a $ spanned by almost finite syntomic $A$-algebras.   
\end{theorem}
\begin{proof}
By the definition of almost finitely generated projective modules and Proposition~\ref{almost-cotangent}, the category $\Alg_{\tm \otimes_{V} A}^{\rm FSyn}$ is equivalent to the pullback of $\Alg^{\rm alFSyn}_A $ along the categorical equivalence n the functor $(V \oplus_{\tm} \tm \otimes_{V} (-))^a: \CAlg_{\tm \otimes_{V} A} \to \CAlg_{A}^a $.  \qed
\end{proof}

Fix a regular cardinal $\kappa$.
We write $\MS=\MS_{\rm FSyn}$ and let $\MS^\kappa$ denote full
subcategory spanned by all $\kappa$-small objects of $\MS$. Referring the
results~\cite[Theorem 3.4.1 and Lemma 3.5.1]{ModMGL} and \cite[Lemma
5.1.3]{MR4325285}, we define the almost version of algebraic cobordism:
\begin{definition}
Let $A$ be an almost $V$-algebra. Let $\alFSyn(A)$ denote the largest
groupoid of the category $\Alg_A^{\rm alFSyn}$, and set $\AMGL=
\Sigma_+^\infty \alFSyn^\kappa(-)$, where $\alFSyn^\kappa$ is the right
Kan extension of the restriction $\alFSyn|_{(\MS^\kappa)^{\rm op}}$
along the inclusion $(\MS^\kappa)^{\rm op} \to \MS^{\rm op}$. We say
that the motivic spectrum $\AMGL$ is the {\it almost algebraic
cobordism}.
\end{definition}

For an almost $V$-algebra $A$, let $\Alg_A^{\rm alFSyn}$ denote the
category of almost finite syntomic $A$-algebras. The results in
Section~\ref{sec-almostDerivedNakayama} provides us the following:
\begin{proposition}
\label{lim-FSyn}Let $A$ be an almost $V$-algebra and $I$ a tight ideal
which is contained in the Jacobson radical of $A$.  For $n \ge 1$, set
$A_n =A/I^n$, and $A_\infty=\varprojlim A_n$. Let $(A_n)_{ n\ge
1}$ be an inverse system of commutative rings and write $A=\varprojlim
A_n$.  The functor $(- \otimes_A A_n)_{ n \ge 1} : \Alg_{A}^{\rm alFSyn}
\to 2-\varprojlim_{n} \Alg_{A_n}^{\rm alFSyn}$ is a categorical
equivalence.
\end{proposition}
\begin{proof}
Let $(B_n)$ be an inverse system of almost finite syntomic
$A_n$-algebras. Then, by the assumption, $B=\varprojlim A_n$ is almost
finite, projective $A$-algebra.  We show that the relative cotangent
complex $L_{B/A}$ is perfect and tor-amplitude in $[-1,\,0]$. Since $B
\to B_n$ is unramified, the sequence $A \to B \to B_n$ induces a
quasi-isomorphism $\tm \otimes_V L_{B/A} \otimes^\mathbb{L}_B B_n \to
\tm \otimes_V L_{B_n/ A} $. Similarly, the sequence $A \to A_n \to B_n$
induces a quasi-isomorphism $\tm \otimes_V L_{B_n/A} \to \tm \otimes_V
L_{B_n /A_n}$. Therefore $ \tm \otimes_V L_{B/A} \otimes^\mathbb{L}_B
B_n \to \tm \otimes_V L_{B_n/A_n}$ is a quasi-isomorphism. Since each
$B_n$ is an almost finitely generated projective $A$-module, $\tm
\otimes_V B \otimes^\mathbb{L}_A A_n \to \tm \otimes_V B_n$ is an
isomorphism for each $n$. Then $\tm \otimes_V L_{B/A}
\otimes^\mathbb{L}_A A_n \to \tm \otimes_V L_{B/A} \otimes^\mathbb{L}_B
(B \otimes_A A_n ) \to \tm \otimes_V L_{B/A} \otimes^\mathbb{L}_B B_n $
is a composition of quasi-isomorphisms for each $n$. By
Theorem~\ref{lim-perfect}, $\tm \otimes_V L_{B/A} \to \varprojlim \tm
\otimes_V L_{B_n/A_n}$ is a quasi-isomorphism and they are almost
perfect complexes of $B$-modules.
By the Milnor exact sequence
\[
  0 \to \varprojlim {}^1 H_{\bullet+1}(  \tm \otimes_V  L_{B_n/A_n} ) ) \to 
 H_{\bullet}(\varprojlim  \tm \otimes_V  L_{B/A} ) \to  
\varprojlim  H_{\bullet}( \tm \otimes_V  L_{B_n/A_n}   ) \to 0,
 \]   
the projective limit $\tm \otimes_V  L_{B_n/A_n} $ is tor-amplitude in $[-1,\,0]$ in $D(\Mod_{B_{!!}} )$.  \qed
\end{proof}

\begin{corollary}
\label{lim-AFSyn}Let $A$ be an almost $V$-algebra and $I$ a tight ideal
which is contained in the Jacobson radical of $A$.  
Let $(A_n)_{ n\ge 1}$ be an inverse system of commutative rings and write $A=\varprojlim A_n$.   The induced map $\alFSyn^\kappa(\Spec A) \to 
\alFSyn^\kappa(\varinjlim \Spec A_n)$ is a weak equivalence of spaces.
\end{corollary}
\begin{proof}
We may assume that the functor $\Spec{A}$ is small. By Proposition~\ref{lim-FSyn}, one has a chain of weak equivalences:
\begin{multline*}
  \alFSyn(\Spec A) \simeq \Map_{\mathrm{Cat}}(\Delta^0,\, \alFSyn_A) \simeq   \Map_\mathrm{Cat}(\Delta^0,\, 2-\varprojlim \alFSyn_{A_n}) \\ \simeq  \varprojlim \Map_\mathrm{Cat}(\Delta^0,\, \alFSyn_{A_n}) \simeq  \varprojlim \alFSyn(\Spec A_n), 
\end{multline*}
where $\Cat{}$ denotes the large category of small categories that
mapping spaces are the largest groupoids of the functor
categories. Therefore, one has an equivalences: $\alFSyn^\kappa(\Spec A)
\simeq \varprojlim \alFSyn^\kappa(\Spec A_n) \simeq \alFSyn^\kappa(\varinjlim \Spec A_n)$.  \qed
\end{proof}

\begin{remark}
Proposition~\ref{lim-FSyn} and its corollary holds under the only the assumption that $\varprojlim: 2-\varprojlim \APMod_{A_n} \to \APMod_A $ is a categorical
equivalence. 
\end{remark}

Let $V$ be a commutative ring and $\omega \in V$ a non-zero divisor. For
$n \ge 0 $, set $A_n= A / \omega^{n+1} A$. Write $\overline{A_n}=
\omega^n A / \omega^{n+1} A$. Since $ \overline{A_n}$ is square-zero
non-unital ring, the $A$-linear inclusion: $ \overline{A_n} \to A_n$ is
a ring homomorphism as non-unital rings, inducing a ring homomorphism:
\begin{equation}
\label{n!!}  
  \varphi_n: (\overline{A_n})_{!!} \to  (A_{n})_{!!} 
\end{equation} 
Hence $(A_{n})_{!!}$ has an $(\overline{A_n})_{!!} $-algebra structure. 

By the canonical isomorphism $ \overline{A_1} \to  \overline{A_n}$, the following holds. 
\begin{lemma}
\label{n-to-1} The linear map $(\omega^{n-1} \cdot \mathrm{Id})_{!!}:
(A_{n})_{!!} \to (A_{n})_{!!}  $ induces an isomorphism:
$(\overline{A_1})_{!!} \to (\overline{A_n})_{!!}$ of $A_{!!}$-algebras.
\qed
\end{lemma}


\subsection{The almost algebraic cobordism of perfectoid algebras}
In this section, we consider the case that $\mathfrak{m}$ is a flat
$V$-module, entailing the canonical map $\mu_{\mathfrak{m}}: \tm \to
\mathfrak{m}$ is bijective. 
We recall the definition of perfectoid algebra. 
\begin{definition}
Let $K$ be a complete non-Archimedian non-discrete valuation field of
rank $1$, and $K^\circ$ denote the subring of powerbounded elements.  We
say that $K$ is a {\it perfectoid field} if the Frobenius $\Phi:
K^\circ/ p \to K^\circ /p $ is surjective, where $p$ is a positive prime
integer which is equal to the characteristic of the residue field of
$K^\circ$.
\end{definition} 

In this section, we fix a perfectoid field $K$ whose valuation ring
$K^\circ$ is mixed characteristic $(0,\,p)$. We put $V=K^\circ$ and
$\mathfrak{m}=K^{\circ \circ}$, where $K^{\circ \circ}= \{x \in K \mid
|x| < 1 \}$ is the maximal ideal of $K^\circ$. Then $\mathfrak{m}$ is
the set of topologically nilpotent elements, being an idempotent
ideal. We fix a pseudouniformizer $\omega \in V$ with $|p| \le |\omega|
< 1$. In this case, the ideal $\mathfrak{m}= \varinjlim_{n \ge 1}
\omega^{1/p^n}V$, which is a filtered colimit of free $V$-modules, is
flat by Lazard's theorem.

\begin{definition} 
An integral perfectoid $V$-algebra is an $\omega$-adic complete flat
$V$-algebra $A$ on which Frobenius induces an isomorphism $\Phi: A/
\omega^{\frac{1}{p}}A \to A / \omega A $.
\end{definition}

For any $V$-algebra $B$,let $B^\flat$ denote the tilting algebra $\varprojlim_{ x
\mapsto x^p} B/ \omega B$ of $B$. The tilting ideal $\mathfrak{m}^\flat \subset V^\flat $ is a flat
$V^\flat$-module as $\mathfrak{m}$ is. Therefore, we can apply almost
mathematics to perfectoid algebras.

\begin{proposition}
\label{n-AFSyn} Let $K$ be a perfectoid field with the valuation ring
 $V$ whose residue field is of characteristic $p>0$. Let $\omega$ be a
 pseudouniformizer and $\mathfrak{m}=\varinjlim_{m \ge 1}
 \omega^{p^{-m}} V$. Let $A$ be an integral perfectoid $V$-algebra and
 write $A_n=A /\omega^n A$. Then $(A_{n})_{!!}$ is an ind-finite
 syntomic $(\overline{A_1})_{!!}$-algebra for $n \ge 1$.
\end{proposition}
\begin{proof}
We may assume $A_*=A$ and $A_!= \mathfrak{m} \otimes_V A$.  
The injection $ (\overline{A_n})_{!} \to (A_{n})_{!}$ is induced by the
inductive system of (non-unital) ring homomorohisms:
\[
      \omega^{n+ \frac{1}{p^m}}A/  \omega^{n+1+ \frac{1}{p^n}}A \to   \omega^{\frac{1}{p^n}}A/  \omega^{n+1+ \frac{1}{p^m}}A,   
\]
which induces a finite syntomic homomorphism 
\[
 \varphi_{n,\,m}: V \oplus (\omega^{\frac{1}{p^m}})^{n p^m+1} A/  \omega^{n+1+ \frac{1}{p^n}}A \to V \oplus 
      \omega^{ \frac{1}{p^m}}A/  \omega^{n+1+ \frac{1}{p^n}}A 
\] 
for each $m \ge 1$. Since the functor $(-)_{!!}$ preserves all small
colimit, by definition of almost relative cotangent complex,
$(A_{n})_{!!}$ is ind-finite syntomic over $ (\overline{A_1})_{!!}$ by
Lemma~\ref{n-to-1} for each $n \ge 1$. \qed

\end{proof}

\begin{lemma}
\label{lemmaA}
On the condition of Proposition~\ref{n-AFSyn}, let $A_n^+$ denote the
cokernel of
\[
   \mathfrak{m}/ \omega^n \mathfrak{m} \to
 V/\omega^n \oplus (\mathfrak{m} \otimes_V A/\omega^n A)    
\]
for each $n$. Then there exists a  canonical isomorphism  $(A_n)_{!!} \to A_n^+$  for each $n \ge 1$.
\end{lemma}
\begin{proof}
Consider the following commutative diagram:
\[
  \xymatrix@1{
\mathfrak{m} \ar[d] \ar[r] &\mathfrak{m} /\omega^n \mathfrak{m} \ar[r] \ar[d] & \mathfrak{m} \otimes_V A_n \ar[d] \\
V \ar[r] &  V /\omega^n  \ar[r] &  A^+_n,  
 }
\]
where both squares are coCartesian. Therefore the morphism $ (A_n)_{!!} \to A^+_n $ is canonically an isomorphism. \qed  
%
\end{proof}

Set $\alFSyn^\flat(A)=\alFSyn(A^\flat)$ for any perfectoid algebra $A$, and write $\AMGL^\flat=\Sigma_+^\infty \alFSyn^\flat$. 
Finally, we prove the tilting equivalence of almost algebraic
cobordism:
\begin{theorem}
\label{limit-1} Let $V$ be a mixed characteristic perfectoid valuation
ring with unit and $V^\flat$ denote the tilting. Then these spectra
$\AMGL$ and $\AMGL^\flat$ are equivalent on zero-section stable affine
integral perfectoid schemes over $V$. 
\end{theorem} 
\begin{proof}
By Lemma~\ref{lemmaA}, there exists a zig-zag of isomorphisms 
\[
 (A_1)_{!!} \rightarrow  (A^+_1)_{!!}   \leftarrow (A^\flat_1)^+_{!!}  \leftarrow (A_1^\flat)_{!!}
\]
of commutative rings, where the middle map is a canonical isomorphism.
 By Proposition~\ref{n-AFSyn}, for each $n \ge 1$, $\AMGL(\mathrm{Z}_0
 (A_n) )$ (resp. $\AMGL(\mathrm{Z}_0(A^\flat_n) )$) is weakly equivalent
 to the same $ \AMGL(A_1 )$ (resp. $ \AMGL(A^\flat_1)$), implying that
 $\mathbb{R}\varprojlim \AMGL(\mathrm{Z}_0( A_n) )$ and
 $\mathbb{R}\varprojlim \AMGL(\mathrm{Z}_0( A^\flat_n ))$ are weakly
 equivalent, where $\mathrm{Z}_0$ denotes the zero-section
 stabilization. Note that the zero-section stabilization preserves small
 colimits. The zero-section stability of $ \varinjlim \Spec A_n $
 implies that one has a chain of equivalences $ \mathbb{L} \varinjlim \Spec A_n \simeq \mathrm{Z}_0  (\mathbb{L} \varinjlim \Spec A_n) \simeq \mathbb{L} \varinjlim \mathrm{Z}_0 ( \Spec A_n)$.  Similarly, one has  $ \mathbb{L} \varinjlim \Spec A^\flat_n \simeq \mathbb{L} \varinjlim \mathrm{Z}_0 ( \Spec A^\flat_n)$.  
Consider the diagram of fiber sequences:
 \[ \xymatrix@1{ \AMGL(\mathbb{L} \varinjlim \Spec A_n ) \ar[r]\ar[d] &
 \AMGL( \oplus_n \mathrm{Z}_0(\Spec A_n) ) \ar[d]
 \ar[r]<0.5mm>^{\mathrm{Id}} \ar[r]<-0.5mm>_{ (\oplus t_n)^* } & \AMGL(
 \oplus_n \mathrm{Z}_0(\Spec A_n ) ) \ar[d] \\ \mathbb{R}\varprojlim
 \AMGL( \mathrm{Z}_0 (A_n) ) \ar[r] & \prod_n \AMGL(\mathrm{Z}_0(A_n) )
 \ar[r]<0.5mm>^{\mathrm{Id}} \ar[r]<-0.5mm>_{\prod t_n } & \prod_n
 \AMGL(\mathrm{Z}_0(A_n) ), }
\]
where the middle and right vertical arrows are weak equivalences.  Since
each $\Spec{A_n}$ is cofibrant and closed immersion $t^*_n:\Spec{A_n}
\to \Spec{A_{n+1}} $ a cofibration with respect to the motivic model
structure, the inductive system $( \Spec{A_n})_{n \ge 1}$ is injectively
cofibrant. Therefore $\varinjlim \Spec{A_n} \to \mathbb{L} \varinjlim
\Spec A_n$ is a motivic equivalence. Furthermore, by
Proposition~\ref{lim-FSyn}, one has weak equivalences of spectra
$\alFSyn^\kappa( \Spec A) \simeq \alFSyn^\kappa( \varinjlim \Spec A_n)$ and
$\alFSyn^\kappa(\Spec A^\flat) \simeq \alFSyn^\kappa( \varinjlim \Spec
A^\flat_n)$. Hence $\AMGL  $ and $\AMGL^\flat  $ are weakly equivalent on the category of integral perfectoid $V$-algebras. \qed
\end{proof}

%
%
%
%
\newcommand{\etalchar}[1]{$^{#1}$}

\bibliographystyle{alphadin}

\nocite{MR2276611}
\nocite{MR2034012}
\nocite{DI}
%
\nocite{DK,DK3,DK2}
%
\nocite{MR4092773}
%
%
%
%
\nocite{GeSn}
%

\nocite{GJ}
\nocite{Hu}
%
\nocite{Jardine, Jardine2}
\nocite{Joyal}
%
%
%
\nocite{HT,HA} 
%
%
\nocite{zbMATH05115214}
%
%
\nocite{zbMATH05550622}
\nocite{MacLane-Moerdijk}
%
%
%
\nocite{MV}
\nocite{Nuiten}
%
%
\nocite{SpO}
\nocite{PPR}
%
\nocite{QuillenHomotopy}
\nocite{Quillen}
%
%
%
%
%
%
\nocite{SSpec}
\nocite{stacks-project}
%
%
%
%
\nocite{VoeH}
\nocite{V}
\nocite{W1}
\nocite{weibel2}
\end{document}